\documentclass[11pt, reqno]{amsart}
\usepackage{enumerate}
  \usepackage{geometry }
 \usepackage{amsmath}
 \usepackage{amsfonts}
 \usepackage{amssymb}
  \usepackage{enumerate}
 
\usepackage{amssymb,amsmath,amscd}
\usepackage{amsthm}

\usepackage{color}
 \usepackage{mathtools}
 \usepackage{enumerate}

 \DeclareMathOperator{\Ind}{Ind}
 \numberwithin{equation}{section}

\newtheorem{theorem}{Theorem}[section]
 
\newtheorem*{theorem*}{Theorem }
\newtheorem{proposition}[theorem]{Proposition}

\newtheorem{conv} {Convention}[section]
\newtheorem{lemma}[theorem]{Lemma}

\DeclareMathOperator{\End}{End}

\DeclareMathOperator{\tr}{tr}

 \usepackage{mathtools}
\usepackage[bbgreekl]{mathbbol}
\usepackage{amsfonts}

\DeclareSymbolFontAlphabet{\mathbb}{AMSb}
 \usepackage[
 pagebackref,  
 colorlinks=true,  citecolor=cyan,   
 citebordercolor={0 .255 .255} , 
  linkbordercolor={0 .255 .255},  
  linkcolor=cyan
 ]{hyperref}

   \usepackage{cite}

 \def\equationautorefname~#1\null{(#1)\null}
 
\usepackage{multirow}
  \usepackage{graphicx}
 \usepackage{caption}
 \usepackage{booktabs}
 \usepackage{threeparttable}
\newcommand{\C}{\mathbb C}
\newcommand{\R}{\mathbb R}

\renewcommand{\tr}{\mathbf{tr}}

      \DeclareMathOperator{\SO}{SO}
      \DeclareMathOperator{\cl}{C\ell}  
\DeclareMathOperator{\ccl}{\mathbb{C}\ell} 
\DeclareMathOperator{\Spin}{Spin}

\usepackage{slashed}
\newcommand{\fsl}[1]{{\slashed{#1}}} 

      \usepackage[normalem]{ulem}
      \usepackage{color}
        \usepackage{xcolor}
      \usepackage{cancel}

\newcommand\blfootnote[1]{%
  \begingroup
  \renewcommand\thefootnote{}\footnote{#1}%
  \addtocounter{footnote}{-1}%
  \endgroup
}
\begin{document}

\title[On Poisson transform for spinors]
{ On Poisson transforms for spinors}
\author{Salem Bensaïd}
\address{Salem Bensaïd : 
Mathematical Sciences Department, College of Science, United Arab Emirates University, Al Ain, UAE
}
\email{salem.bensaid@uaeu.ac.ae}
\thanks{}

\author{Abdelhamid Boussejra}
\address{Abdelhamid Boussejra : Department of Mathematics, Faculty of Sciences, University Ibn Tofail, Kenitra, Morocco 
}
\curraddr{}
\email{boussejra.abdelhamid@uit.ac.ma}
\thanks{}

\author{Khalid Koufany}
\address{Khalid Koufany : Université de Lorraine, CNRS, IECL, F-54000 Nancy, France}
\curraddr{}
\email{khalid.koufany@univ-lorraine.fr}
\thanks{}
\date{\today}
 
 \maketitle

\begin{abstract} 
Let $(\tau,V_\tau)$ be a spinor representation of $\Spin(n)$ and let $(\sigma,V_\sigma)$ be a spinor representation of $\Spin(n-1)$  that occurs in the restriction $\tau_{\mid \Spin(n-1)}$.  

We  consider the real hyperbolic space $H^n(\mathbb R)$ as the rank one homogeneous space $\Spin_0(1,n)/\Spin(n)$ and the spinor bundle $\Sigma H^n(\mathbb R)$ over $H^n(\mathbb R)$ as the homogeneous bundle $\Spin_0(1,n)\times_{\Spin(n)} V_\tau$.

Our aim is to characterize eigenspinors  of the algebra of invariant differential operators acting on $\Sigma H^n(\mathbb R)$ which can be written as the Poisson transform of  $L^p$-sections   
of the  bundle  $\Spin(n)\times_{\Spin(n-1)} V_\sigma$ over  the boundary $S^{n-1}\simeq \Spin(n)/\Spin(n-1)$ of $H^n(\mathbb R)$, for $1<p<\infty$.


  \end{abstract}
\blfootnote{\emph{2010 AMS Classification : \rm 15A66, 43A85, 22E30}} 

\tableofcontents


%
%
%
 
 \section{Introduction}
 
 In this article we go on with the investigation of Poisson transforms on vector bundles over the real hyperbolic space $H^n(\mathbb R)$ which we started  in \cite{BBK} for the case of the bundle of differential forms. Since the Clifford algebra of $\mathbb R^n$ can be realized as the exterior algebra $\bigwedge \mathbb R^n$ (via the quantization map), it becomes   natural   to extend the results of \cite{BBK} to the case of the spinor bundle over $H^n(\mathbb R)$.
To be more explicit,  we shall realize $H^n(\mathbb R)\simeq G/K=\Spin_0(1,n)/\Spin(n)$  as the unit ball in $\mathbb R^n$ and let $S^{n-1}=\partial H^n(\mathbb R)\simeq K/M=\Spin(n)/\Spin(n-1)$ be its boundary. 
Let $\tau_n$ be     the complex spin representation of $\Spin(n)$,
which is irreducible
when $n$ is odd and splits into two irreducible components $\tau_n=\tau_n^+\oplus\tau_n^-$ (half-spin representations) when $n$ is even. 
   It is known by classical arguments that, if $n$ is even, then each restriction ${\tau_n^\pm}_{|\Spin(n-1)}$ is a unitary irreducible 
   representation of $\Spin(n-1)$, and therefore coincides with the spin representation  of $\Spin(n-1)$ which we will denote  by $\sigma_{n-1}$. Note that $\sigma_{n-1}=\tau_{n-1}$ and we will use in the sequel  the letter $\tau$ (resp. $\sigma$) for the spin or half-spin representations of  $K=\Spin(n)$ (resp.$M=\Spin(n-1)$). On the other hand, if $n$ is odd, then ${\tau_n}_{|\Spin(n-1)}$ splits into two irreducible representations of $\Spin(n-1)$. More explicitly,  ${\tau_n}_{|\Spin(n-1)}=\sigma_{n-1}^+\oplus \sigma_{n-1}^-$, where $\sigma_{n-1}^\pm$ are the half-spin representations of $\Spin(n-1)$.
   
   To make the notations in this introduction less cumbersome, we let $\tau=\tau_n$ or $\tau_n^\pm$ whether $n$ is odd or even and we put $\widehat M(\tau)=\{\sigma_{n-1}^+,\sigma_{n-1}^-\}$ or $ \{\sigma_{n-1}\}$  accordingly.

    For $\lambda$ in $\mathbb C$ and $\sigma\in \widehat M(\tau)$, we consider the Poisson transform
    \begin{align*}
\mathcal{P}_{\sigma,\lambda}^\tau \colon C^{-\omega}(K/M,\sigma)\rightarrow C^\infty(G/K,\tau)
\end{align*}
given by
\begin{equation*}
\mathcal{P}_{\sigma,\lambda}^\tau\, f(g)=   \fsl\kappa  \int_K {\rm e}^{-(i\lambda+\rho)H(g^{-1}k)}\tau(\kappa(g^{-1}k)) \iota^\tau_\sigma(f(k)){\rm d}k
\end{equation*} 
where $\iota_\sigma^\tau$ is the embedding of $\sigma$ in $\tau$ and  $\fsl\kappa$ is the constant given by \autoref{kappa}. 
Above $C^{-\omega}(K/M,\sigma)$ is the space  of hyperfunction sections of the bundle $K\times_M V_\sigma$  viewed as  the space
of $V_\sigma$-valued covariant hyperfunctions   on $K$ and 
    $C^\infty(G/K,\tau)$ is the space of smooth sections of the bundle $G\times_K V_\tau$ viewed as the space  of $V_\tau$-valued smooth covariant functions on $G$. 
    
    It is proved by Gaillard \cite{Gaillard}  that the commutative  algebra $D(G,\tau)$ of $G$-invariant differential operators acting on $C^\infty(G/K,\tau)$ is generated explicitly as
    $$\begin{cases}
    D(G,\tau^{\pm}_n)\simeq \mathbb{C}[\fsl D^2] & \text{if $n$ is even}\\
  D(G,\tau_n)\simeq \mathbb{C}[\fsl D] & \text{if $n$ is odd}
    \end{cases}
    $$
    where $\fsl D$ is the Dirac operator. Moreover, it is well known (see e.g. Camporesi and Pedon  \cite{CP}, Olbrich \cite{olbrich}) that  the Poisson integrals 
    $\mathcal P_{\sigma,\lambda}^\tau f$, for $f\in C^{-\omega}(K/M,\sigma)$,  are eigenfunctions of $D(G,\tau)$.
         
   The main result of this paper is to characterize   the image  of the spaces $L^p(K/M,\sigma)$, for $1<p<\infty$,  by  the Poisson transform $\mathcal{P}_{\sigma,\lambda}^\tau$. 
    To do so,  we introduce 
   the spaces $\mathcal{E}_{\sigma,\lambda}^p(G/K,\tau)$ consisting  of eigensections of $D(G,\tau)$ satisfying a   Hardy-type growth condition,     see \autoref{hardy-type-even} and \autoref{hardy-type-odd} for the precise definition.
      \begin{theorem*}[Main result]\label{main-resu}
   Let $1<p<\infty$ and let $\lambda$ in $\mathbb C$ such that $\Re(i\lambda)>0$.
   
$(1)$ If $n$ 	is even,   $\mathcal{P}_{\sigma_{n-1},\lambda}^{\tau_n^\pm}$ is a topological isomorphism of the space $L^p(K/M, \sigma_{n-1})$ onto the space $\mathcal{E}_{\sigma_{n-1},\lambda}^p(G/K,\tau_n^\pm)$.

$(2)$ If $n$ 	is odd,   $\mathcal{P}_{\sigma_{n-1}^\pm,\lambda}^{\tau_n}$ is a topological isomorphism of the space $L^p(K/M, \sigma_{n-1}^\pm)$ onto the space $\mathcal{E}_{\sigma_{n-1}^\pm,\lambda}^p(G/K,\tau_n)$.
   \end{theorem*}
   
  
  The paper is organized as follow. In section 2 we recall some useful facts on Clifford modules and spinor representations. 
  In section 3 we define the principal series representations of the Lie group $\Spin_0(1,n)$ and in section 4 we set up  the Poisson transform.
In the remaining sections we prove     the main result where we will follow the same steps as in \cite{BBK}. More precisely, in section 5 we prove a Fatou-type theorem, which will allows us, in particular, to compute the vector-valued Harish-Chandra $c$-function associated with $\tau$ in terms of Gamma functions.   The    Fatou-type theorem is essentially used in section 6 to establish   the main result  for $p=2$. In the same section an inversion formula for Poisson transforms is also proved. Section 7 is devoted to the proof  of  the main result  for every $1<p<\infty$. The proof is based on  a reduction argument to the case $p=2$ followed by the use of the inversion formula.

\section{The spinor representations}

Consider the Euclidean vector space $\R^n $ equipped  with the standard  inner product 
\begin{equation}\label{inner}
\langle x,y\rangle =\sum_{i=1}^n x_iy_i.\end{equation}
Choose an orthonormal basis  $\{e_1,\cdots, e_n\}$ on $\mathbb R^n.$
The \emph{Clifford algebra}  $\cl(n)$   is the algebra over $\mathbb R$ generated by the vector space $\R^n$ and the relations
\begin{equation}\label{rel-clif}
xy +yx=-2\langle x,y\rangle\qquad \text{for } x, y \in \R^n,
\end{equation}
where $-2\langle x,y\rangle$ is identified with $-2\langle x,y\rangle 1$ and $1$ being the algebra identity element of   $\cl(n)$. 
In particular, we have 
$$e_i^2=-1\; \text{ and }\; e_ie_j=-e_je_i \; \text{ for } \; i\not=j.$$

As a real vector space, $\cl(n)$ has a basis consisting of  
 $$e_\emptyset=1, \quad e_I=e_{i_1}e_{i_2}\cdots e_{i_k}$$
 with $I=\{i_1,\ldots, i_k\}\subset \{1,\ldots, n\}$ and $i_1<i_2<\cdots <i_k$. 

 We denote by $a'$, $a^*$ and $\bar a$ the conjugations of any $a\in \cl(n)$. By linearity, 
 it is sufficient to define the effect on the basis elements $e_{I}=$ $e_{i_{1}} e_{i_{2}} \cdots e_{i_{k}}.$ The conjugations are
\begin{align*}
 e_{I} \mapsto e_{I}^{\prime}  &= (-1)^{k} e_{I}   &&\text{(main involution)}\\
  e_{I} \mapsto e_{I}^{*}  & =  (-1)^{\frac{k(k-1)}{2}} e_{I} &&\text{(reversion)}\\
  e_{I} \mapsto \bar{e}_{I}  &  = (-1)^{\frac{k(k+1)}{2}} e_{I}  &&\text{(Clifford conjugation).}
 \end{align*}
The later conjugation   will be denoted by $\alpha(a):=\bar{a}$. 
%
Observe that $a \mapsto a^{\prime}$ is an involution $\left((a b)^{\prime}=a^{\prime} b^{\prime}\right)$ while the conjugations $a \mapsto a^{*}$ and $a \mapsto \bar{a}={a^*}'$ are anti-involutions $\left((a b)^{*}=b^{*} a^{*}, \overline{a b}=\bar{b} \bar{a}\right)$. In particular
\begin{equation*}
(e_{j_1}\cdots e_{j_k})(e_{j_1}\cdots e_{j_k})'=
\begin{cases} 
1& \text{ if $k$ is even}\\ 
-1& \text{ if $k$ is odd}
\end{cases}
\end{equation*}
Under the main involution, the Clifford algebra splits into a direct sum of {\it even} and {\it odd} elements:
\begin{equation*}
  \cl(n) =\cl^+(n)\oplus \cl^-(n). 
\end{equation*}

For $a\in  \cl(n)$ define the norm  $|a|=\alpha(a)a=\bar a a$. In particular, if $a\in\R^n$, then $\bar a=-a$ and from \autoref{rel-clif} we get  $|a|=\langle a,a\rangle$.


 The spin groupe $\Spin(\mathbb R^n)$ is the group of elements  in   $\cl(n)$   
 of the form
 $$g=x_1\cdots x_{2k}, \;\;  x_i\in \R^n, \; |x_i|=1\; \text{ for } i=1,\ldots, 2k.$$
 In particular,  $\Spin(\R^n)$ is a Lie group which is a twofold covering of $\SO(n)$. 
 We will often write  $\Spin(n)$ in place of   $\Spin(\R^n)$. 
 
Denote by  $\ccl(n)$  the complex Clifford algebra of  $\mathbb C^n$, which can be identified with   $\cl(n)\otimes \mathbb C$.

A \emph{Clifford module} $(\tau, \mathbb S_n)$ is a complex vector space $\mathbb S_n$   together with an  action $\tau$ of  $\ccl(n)$ on $\mathbb S_n$.  As $\Spin(n)\subset \cl(n)\subset \ccl(n)$, the representation   $\tau$ of $\ccl(n)$ restricts to a representation of $\Spin(n)$. 
 
 


The Clifford modules  $\mathbb S$ have different realizations.  In the sequel we will use the {\it Lagrangian space}  realization, where we shall distinguish the difference between an even dimension and an odd dimension. For more details, we refer the reader to \cite{del-al, LM}.

\subsection{The even-dimensional case}\label{?.?}
Assume that $n=2m.$ In terms of the orthonormal basis $e_1,\ldots, e_{2m}$ of $\R^{2m}$, 
define for each $j=1,\ldots, m$, 
\begin{eqnarray*}
f_{j}&=&
  \frac{1}{\sqrt{2}}\left(e_{2 j-1}+i e_{2 j}\right),\\
 \alpha(f_j)=\bar f_{j}&=& 
 -\frac{1}{\sqrt{2}}\left(e_{2 j-1}-i e_{2 j}\right).
\end{eqnarray*}
Set 
$$ W_m=\mathrm{span}(f_1,\ldots, f_m)\; \text{ and } \, 
\overline{W}_m=\mathrm{span}(\bar{f}_1,\ldots, \bar{f}_m).$$
Then $W_m$ is a maximal isotropic subspaces of $\mathbb C^{n}$,
$$\mathbb C^n=W_m \oplus \overline{W}_m,$$
and  the inner product \autoref{inner} extends to $\ccl(n)$ and  induces a non-degenerate duality between $W_m$ and $\overline W_m$.
 
 
 In these circumstances, the {\it Dirac spinor space} $\mathbb S$ is defined as the exterior algebra $\bigwedge W_m$, that is $\mathbb S= \bigwedge W_m$. If we want to emphasize the dimension   we write $\mathbb S_n$ or $\mathbb S_{2m}$.

 There is a unique $\ccl(n)$-module  structure       on $\mathbb S_{2m}=\bigwedge W_m$ such that   
\begin{equation*}
 \tau(u+\bar v)x= {\sqrt{2}} {\boldsymbol \varepsilon}(u) x- {\sqrt{2}}{\boldsymbol \iota}(\bar v) x,   
\end{equation*}
for  $u\in W_m$, $v\in \overline{W}_m$ and  $x\in \mathbb S_{2m}$. 
Above, ${\boldsymbol \varepsilon}(u)$ and ${\boldsymbol \iota}(\bar v)$ denote, respectively,  the exterior and the  interior products on $\ccl(2m)$.
The action $\tau$ gives rise to a representation of the Clifford algebra $\ccl(2m)$, which we will   denote  by $\tau_{2m},$  and called the \emph{spinor representation}.  The representation $(\tau_{2m}, \mathbb S_{2m})$ is of dimension $2^m$ and, up to equivalence,  it is the unique irreducible representation of $\ccl(2m)$.



Let  $\mathbb S^+_{2m}=\bigwedge^+W_m$ be the even part and $\mathbb S^-_{2m}=\bigwedge^-W_m$ be the odd part of $\bigwedge W_m$, which are  stable under the action of  $\ccl^+(n)$. As $\Spin(2m)\subset \ccl^+(2m),$ the representation $\tau_{2m}$  splits as a direct sum of  two  non-equivalent irreducible representations $(\tau_{2m}^\pm,\mathbb S_{2m}^\pm)$  of $\Spin(2m)$, called the {\it half-spinor representations}. Both representations are of dimension $2^{m-1}$.

\subsection{The odd-dimensional case}  Assume that $n=2m-1$. Then we have,  
  $$\mathbb C^n=W_{m-1}\oplus \overline W_{m-1}\oplus \mathbb Ce_{2m-1}$$
  where $W_{m-1}=\text{span}(f_1,\cdots, f_{m-1})$. The spinor space $\mathbb S_{2m-1}=\bigwedge W_{m-1}$ has two module structures over $\ccl(2m-1)$.  Indeed, $W_{m-1}\oplus \overline W_{m-1}$ operates by the same formula as in the even case above while the element $e_{2m-1}$  acts on  $x\in \mathbb S_{2m-1}$  either by multiplication by $i(-1)^{\deg x}$ or  by $-i(-1)^{\deg x}.$    As representations of  $\Spin(2m-1)$,  the above mentioned modules are irreducible,  
equivalent, and thus leading to a unique spinor representation $\tau_{2m-1}$ acting on the space $\mathbb S_{2m-1} =\bigwedge W_{m-1}$ by 
$$\tau_{2m-1}(u+\bar v+\lambda\, e_{2m-1})x=   {\sqrt{2}} \varepsilon(u) x-   { \sqrt{2}} \iota(\bar v) x + i(-1)^{\deg(x)}\lambda \,x$$
for $u\in W_{m-1}$, $\bar v\in \overline W_{m-1}$, $\lambda\in\mathbb C$ and $x\in \mathbb S_{2m-1}$.

  

 \section{Principal series representations of $\Spin_0(1,n)$}
 
 Let $H^{n}(\mathbb{R})$ be the $n$-dimensional real hyperbolic space, $n \geq 2$. We identify $H^{n}(\mathbb{R})$, via the Poincaré model, with the unit ball of $\mathbb{R}^{n}$ and its topological boundary $\partial H^{n}(\mathbb{R})$ with the unit sphere $S^{n-1}$ of $\mathbb{R}^{n}$. We shall realize $H^{n}(\mathbb{R})$ as the rank one symmetric space  $G=\Spin_{0}(1,n)/\Spin(n)$ and     $S^{n-1}$ with the compact symmetric space $\Spin(n)/\Spin(n-1)$.

 Let  $\mathbb R^{1,n}$ be the real vector space of dimension $n+1$ equipped with the symmetric bilinear form 
$$Q(x, y)= x_0y_0-x_1y_1-\dots -x_{n}y_{n},$$
and fix an orthonormal basis  $\{e_0, e_1, \ldots, e_n\}$ of $\R^{n+1}.$ Denote by $\cl(\R^{1,n})$ the corresponding Clifford algebra generated by $\R^{1,n}$ and subject to the relation
$$ x  y+  y x= 2  Q(x, y).$$ 
 As in the previous section, we may define on $\cl(\R^{1,n})$ the main involution, the reversion, and the Clifford conjugaison.
Define the spinorial  norm $\mathcal N$ by 
$$\mathcal N(a)=\alpha(a) a, \quad a\in \cl(\R^{1,n})$$
and 
let $\Spin_0(1,n)$ be the group defined by
$$\Spin_0(1,n)=\{g=x_1\cdots x_{2k}  \; |\; x_j\in\R^{1,n},\; Q(x_j)=\pm1,   \#\{j,\; Q(x_j)=-1\}\; \text{is even} \}.$$
Note that an element  $a\in  \cl(\R^{1,n})$ is in  $\Spin_0(1,n)$ if and only if $a$ is invertible and $\mathcal{N}(a)=1$. Further, the group $G=\Spin_0(1,n)$ is a connected Lie group which turns out to be a twofold covering of $\SO_0(\R^{1,n})=\SO_0(1,n)$ (see \cite{del-al}).
 
 The Lie algebra $\mathbf {\mathfrak g}=\mathfrak o(1,n)$ of $G$ can be realized as the subspace  $\cl^2(\R^{1,n})\simeq \bigwedge^2\R^{n+1}$ of bi-vectors in $\cl(\R^{1,n})$ spanned by $\{e_ie_j, 0\leqslant i<j\leqslant n\}$. 
 %
 %
Let 
\begin{equation}\label{kp}
 \mathfrak k = \bigoplus_{1\leqslant i<j\leqslant n} \mathbb R\,  e_ie_j,\qquad  \mathfrak p = \bigoplus_{j=1}^{n} \,\mathbb R e_0e_j.
\end{equation}
$$\mathfrak m = \sum_{1\leqslant i<j\leqslant n-1} \mathbb R e_ie_j,\qquad  \mathfrak n = \bigoplus_{j=1}^{n-1} \mathbb R e_j(e_0-e_{n}).$$
Moreover, let  $\mathfrak a$ be the Cartan subspace  of $\mathfrak p$  given by
\[\mathfrak a = \mathbb R H,\quad \text{where } H=e_0e_{n} .
\]
%
The analytic Lie subgroups of $G$ associated to the Lie sub-algebras $\mathfrak k$, $\mathfrak a$, $\mathfrak m$ and $\mathfrak n$ will be denoted, respectively, by $K$, $A$, $M$ and $N$. We pin down  that $K\simeq \Spin(n)$ and $M\simeq \Spin(n-1)$. 
 It is known that the homogenous space $G/K$ can be  seen as the  real hyperbolic space $ H^n(\R)$, which we may realize as the open unit ball in $\R^n$. Further,  the homogeneous space $K/M$ can be seen as the boundary $\partial  H^n(\R)$ and realized as the unit sphere $S^{n-1}$ in $\R^n$.
 
Henceforth, we will use the Greek letter $\tau$ to denote the spin representations of $K$ and  the Greek letter $\sigma$ for those of $M$.

Let $(\tau_n, V_{\tau_n})$ be a   unitary complex spin representation of $K=\Spin(n)$, where $V_{\tau_n}=\mathbb S_n$ is  the space of spinors associated with the complexification of $\mathfrak{p}\simeq T_o(G/K)$. 
The {\it spinor bundle}  $\Sigma  H^n(\R) $  
  may be  defined as  the  homogeneous vector bundle
  $$\Sigma  H^n(\R) \equiv G\times_K \mathbb S_n$$ over $ H^n(\R)$ associated  with  $\tau_n$.

 The space $C^\infty(\Sigma  H^n(\R))$ of smooth sections   of this bundle will be  identified with the space $C^\infty(G/K,\tau_n)$ of smooth $\mathbb S_n$-valued functions on $G$ such that
$$f(gk)=\tau_n(k^{-1}) f(g)$$
for any $g\in G$ and $k\in K$.






 

Recall the following branching law of $(K,M)=(\Spin(n),\Spin(n-1))$.

\begin{lemma}[See \cite{GW}] Let $\widehat M$ be the set of unitary equivalence classes of irreducible representations of $M.$ Using the same notation as in the previous section, we have: 
\begin{enumerate}
\item[$(1)$]  If $n$ is even,   then   ${\tau_n^\pm}_{|M} = \sigma_{n-1}\in \widehat M$.
\item[$(2)$] If $n$ is odd,   then ${\tau_n}_{|M}=\sigma^+_{n-1}\oplus \sigma^-_{n-1}$ with $\sigma^\pm_{n-1}\in \widehat M$. 
 
  \end{enumerate}
  The decomposing factors occur with multiplicity one. 

\end{lemma}

 



To be more explicit, the above lemma says :
\begin{enumerate}[\upshape 1)]
\item  If $n=2m$, the two half-spin representations $(\tau_{2m}^\pm, \mathbb S_{2m}^\pm)$ of $\Spin(2m)$ are non-equivalent  irreducible  representations. Fix the following identifications $\mathbb S_{2m}^\pm\simeq \bigwedge^\pm \mathbb C^{m}\simeq \bigwedge \mathbb C^{m-1}$.  The restrictions of $\tau_{2m}^\pm$ to $\Spin(2m-1)$ are both isomorphic to $(\tau_{2m-1},\mathbb S_{2m-1})$, where $\mathbb S_{2m-1}$ is identified with $\bigwedge \mathbb C^{m-1}$.
 We will realize the embeddings 
\begin{equation}\label{emb-cas-pair}
\iota_{\sigma_{2m-1}}^{\tau_{2m}^\pm}   \colon (\sigma_{2m-1}, \mathbb S_{2m-1}) \hookrightarrow (\tau_{2m}^\pm, \mathbb S_{2m}^\pm)
\end{equation}
to be    the identity map. 
\item If $n=2m+1$, the spin representation $(\tau_{2m+1},\mathbb S_{2m+1})$ is irreducible. Its restriction to $\Spin(2m)$  
  splits as ${\tau_{2m+1}}_{|\Spin(2m)}=\sigma^+_{2m}\oplus \sigma^-_{2m}$,   where $(\sigma_{2m}^\pm, \mathbb S_{2m}^\pm)$ 
are the two half-spin representations of  $\Spin(2m)$. That is  $\mathbb S_{2m+1}= \mathbb S_{2m}=\mathbb S^+_{2m}\oplus \mathbb S^-_{2m}$.
For any $x=x_1 \wedge \cdots \wedge x_k\in \mathbb S_{2m+1}$ we let
$$\gamma(x_1\wedge  \cdots \wedge x_k)=(-1)^k x_1\wedge   \cdots \wedge x_k.$$
Then the  projection of $(\tau_{2m+1}, \mathbb S_{2m+1})$ onto $(\sigma_{2m}^\pm, \mathbb S_{2m}^\pm)$ is given by
$$\begin{array}{cccl}
\pi_{\sigma_{2m}^\pm}^{\tau_{2m+1}}  & \colon \mathbb S_{2m+1} &\to& \mathbb S_{2m}^\pm \\ 
	& x&\mapsto &\frac{1}{2} (x\pm \gamma(x)),
\end{array}$$
while the embedding  
\begin{equation}\label{emb-cas-impair} 
\iota_{\sigma_{2m}^\pm}^{\tau_{2m+1}} \colon (\sigma_{2m}^\pm, \mathbb S^\pm_{2m})      \hookrightarrow  (\tau_{2m+1}, \mathbb S_{2m+1})
\end{equation}
  is the identity map. One can easily check that the adjoint of $\pi_{\sigma_{2m}^\pm}^{\tau_{2m+1}}$  with respect to the extended inner product $\langle \cdot, \cdot\rangle$ coincides with $\iota_{\sigma_{2m}^\pm}^{\tau_{2m+1}}.$
 \end{enumerate}
\begin{conv}\label{conv} Hereafter we shall use the following convention: 
\begin{itemize}
\item Let $\tau=\tau_n$ or $\tau_n^\pm$. Denote by  $\widehat{M}(\tau)$   the set of representations $\sigma \in \widehat{M}$ that occur in the restriction of $\tau$ to $M$ with multiplicity one.  
\item When $n$ is even, the notation $(\tau, V_\tau)$ refers to the   two half-spin representations $(\tau_n^\pm, \mathbb S_n^\pm),$ and the set $\widehat{M}(\tau)$  reduces to the single $\{\sigma_{n-1}\}.$ 
\item When $n$ is odd, the notation $(\tau, V_\tau)$ refers to the unique spin representation $(\tau_n, \mathbb S_n),$ and the set $\widehat{M}(\tau)$  becomes  $\{\sigma_{n-1}^+, \sigma_{n-1}^-\}.$ 
\end{itemize}
\end{conv}


Let $\tau=\tau_n$ or $\tau_n^\pm$. For $\sigma\in \widehat M(\tau)$ and  $\lambda\in \mathfrak a_{\mathbb C}^*\simeq\mathbb C,$  let $\sigma_{\lambda}$ be   the representation of the parabolic subgroup $ P=MAN$   given by
$$\sigma_{\lambda}(ma_tn):=(\sigma\otimes e^{i\lambda}\otimes 1)(ma_tn)=e^{(\rho-i\lambda) t}\sigma(m),$$
where  $\rho=(n-1)/2$.
The  {\it principal series representation} $\pi_{\sigma,\lambda}$ of $G$ is the associated induced representation  from $P$ to $G$,
$$\pi_{\sigma,\lambda} = \Ind_{P}^{G} \sigma_\lambda.$$

Let $\mathbb S_{\sigma, \lambda}=G\times_P V_{\sigma}=G\times_P \mathbb S_{n-1}$ be the associated homogeneous bundle  over $ G/P$. Its   space of   hyperfonction sections      will be identified  with the space $C^{-\omega}(G/P,\sigma_{\lambda})$ of $V_{\sigma}$-valued hyperfunctions $f$ on $G$ such that
$$f(gma_tn)=e^{(i\lambda-\rho)t} \sigma(m^{-1}) f(g),$$
for any $g\in G$, $m\in M$, $n\in N$ and $a_t\in A$.

By restriction to $K$, $C^{-\omega}(G/P,\sigma_{\lambda})$ is isomorphic (as a $K$-module) to the space $C^{-\omega}(K/M,\sigma)$  of $V_{\sigma}$-valued hyperfunctions $f$ on $K$ satisfying
\begin{equation}\label{cafe}
f(km)=\sigma(m^{-1}) f(k),
\end{equation}
for any $k\in K$ and $m\in M$.  The space $C^{-\omega}(K/M,\sigma)$ can also be seen as the space of hyperfunction sections of the homogeneous bundle $K\times_M V_{\sigma}$ over $K/M$ corresponding to $\sigma$.

Let $G=KAN $ be the Iwasawa decomposition for $G,$ so each  $g\in G $ can be written as 
 $$g= \kappa(g) e^{H(g)} n(g).$$
Then  the compact model of the principal series representation $\pi_{\sigma,\lambda}$ is given by 
$$ \pi_{\sigma,\lambda}(g) f(k) =e^{(i\lambda -\rho) H(g^{-1}k)} f(\kappa(g^{-1} k)),$$
with  $g\in G$ and $k\in K$.

\section{Poisson transforms} 
We continue with the same notations as above. Let  $\iota^\tau_\sigma \colon V_\tau\rightarrow V_\sigma$  be the natural embedding of $V_\sigma$ into $V_\tau$ (see \autoref{emb-cas-pair} for the even case and \autoref{emb-cas-impair} for the odd case). The Poisson transform  on $C^{-\omega}(K/M,\sigma)$ is the map 
\begin{align*}
\mathcal{P}_{\sigma,\lambda}^\tau \colon C^{-\omega}(K/M,\sigma)\longrightarrow  C^\infty(G/K,\tau)
\end{align*}
given by
\begin{equation}\label{Poisson}
\mathcal{P}_{\sigma,\lambda}^\tau\, f(g)=\fsl\kappa  \int_K {\rm e}^{-(i\lambda+\rho)H(g^{-1}k)}\tau(\kappa(g^{-1}k)) \iota^\tau_\sigma(f(k)){\rm d}k,
\end{equation} 
where 
\begin{equation}\label{kappa}
\fsl\kappa   =\sqrt{\frac{\dim \tau}{\dim \sigma}}=\begin{cases}
 1&\text{ if $n$ is even}\\
  \sqrt{2}&\text{ if $n$ is odd.}\\
 \end{cases}	
\end{equation}
 %
%

 Let $D(G,\tau)$ be   the left algebra  of $G$-invariant differential operators acting on $C^\infty(G/K,\tau)$. In  \cite{Gaillard}, the author proved that if $n$ is even, then $D(G,\tau^{\pm}_n)\simeq \mathbb{C}[\fsl D^2],$ and if $n$ is odd, then $D(G,\tau_n)\simeq \mathbb{C}[\fsl D].$ Here $\fsl D$ is the Dirac operator.

The action of the algebra $D(G,\tau)$ on Poisson integrals is described as follows. 

\begin{proposition}[see \cite{CP}]\label{CP} For $\lambda\in \mathbb C,$ we have:

$(1)$  If $n$ is even and $f\in C^{-\omega}(K/M,\sigma_{n-1})$, then  $\mathcal{P}_{\sigma_{n-1},\lambda}^{\tau_n^\pm}f$ is an eigenfunction of   the algebra $D(G,\tau^{\pm}_n)$. More precisely,  we have
\begin{equation*}
\fsl D^2\,\mathcal{P}_{\sigma_{n-1},\lambda}^{\tau_n^\pm} f=\lambda^2 \mathcal{P}_{\sigma_{n-1},\lambda}^{\tau_n^\pm} f.
\end{equation*}

$(2)$   If $n$ is odd and $f\in C^{-\omega}(K/M,\sigma^\pm_{n-1})$, then  $\mathcal{P}_{\sigma_{n-1}^\pm,\lambda}^{\tau_n}$   is an eigenfunction of algebra $D(G,\tau_n)$. More precisely we have 
 \begin{equation*}
\fsl D\, \mathcal{P}_{\sigma_{n-1}^\pm,\lambda}^{\tau_n} f=\mp \lambda\, \mathcal{P}_{\sigma_{n-1}^\pm,\lambda}^{\tau_n} f.
\end{equation*}
\end{proposition}

Notice that our $\lambda$ corresponds to $-\lambda$ in \cite{CP}.

Let us introduce the following eigenspaces: 
 \begin{align} 
\mathcal{E}_\lambda(G/K, \tau_n^\pm)&=\left\{ F\in C^\infty(G/K,\tau_n^\pm) \mid \fsl D^2 F =  \lambda^2 F\right\}&\text{whenever} \,n\, \text{is even},\label{g1}\\
\mathcal{E}_{\pm,\lambda}(G/K, \tau_n)&=\left\{ F\in C^\infty(G/K,\tau_n) \mid \fsl D F =  \mp\lambda F\right\}&\text{whenever} \,n\, \text{is odd}.\label{g2}
\end{align} 

  For $1<p<\infty$ and $\sigma\in\widehat M(\tau)$,
let $L^p(K/M,\sigma)$ be the space of $V_\sigma$-valued functions $f$ on $K$ satisfying \autoref{cafe}
and such that
$$\|f\|_{L^p(K/M,\sigma)}=\left(\int_K \bigl\|f(k)\bigr\|_{\sigma}^p{\rm d}k\right)^{1/p}<\infty.$$
(Here $\Vert \cdot\Vert_\sigma= \Vert \cdot\Vert_{V_\sigma}.$) The main goal  of  the paper is to characterize  the image of  $L^p(K/M,\sigma)$
   by the Poisson transform $\mathcal P_{\sigma,\lambda}^\tau$. To state the main result, let us introduce   the following Hardy type spaces: 
 \begin{align}\label{hardy-type-even}
\mathcal{E}^p_\lambda(G/K,\tau_n^\pm)&=\left\{F \in \mathcal{E}_\lambda(G/K, \tau_n^\pm) \mid \bigl\| F\bigr\|_{\mathcal{E}^p_\lambda}<\infty\right\}\quad \text{whenever} \,n\, \text{is even},
 \end{align}
where
$$\begin{aligned} 
\| F\|_{\mathcal{E}^p_\lambda}&=\sup_{t>0} {\rm e}^{(\rho-\Re(i\lambda))t} \left(\int_K \bigl\| F(ka_t)\bigr\|^p_{{\tau_n}} {\rm d}k\right)^{\frac{1}{p}},
 \end{aligned}
$$
 and
 \begin{align}\label{hardy-type-odd}
 \mathcal{E}^p_{\pm,\lambda}(G/K,\tau_n)&=\left\{F \in \mathcal{E}_{\pm,\lambda}(G/K, \tau_n) \mid \bigl\| F\bigr\|_{\mathcal{E}^p_{\pm,\lambda}}<\infty\right\}\; \text{whenever} \,n\, \text{is odd},
\end{align} 
where
$$\begin{aligned}
 \| F\|_{\mathcal{E}^p_{\pm,\lambda}}&=\sup_{t>0} {\rm e}^{(\rho-\Re(i\lambda))t} \left(\int_K \bigl\| F(ka_t)\bigr\|^p_{{\tau_{n}^\pm}} {\rm d}k\right)^{\frac{1}{p}}.
\end{aligned}
$$

Our main result is 
\begin{theorem}\label{main} Let $1<p<\infty$
and let $\lambda\in \mathbb C$ such that $\Re(i\lambda)>0$.

 $(1)$ If $n$ 	is even, then $\mathcal{P}_{\sigma_{n-1},\lambda}^{\tau_n^\pm}$ is a topological isomorphism of the space $L^p(K/M, \sigma_{n-1})$ onto the space $\mathcal{E}_{\lambda}^p(G/K,\tau_n^\pm)$.
Furthermore, there exists a positive constant $\gamma_\lambda$ such that, for every  $f\in L^p(K/M,\sigma)$ we have
\begin{equation*}
 |c(\lambda,\tau_n^\pm)| \|f\|_{L^p(K/M,\sigma)}
\leqslant
 \bigl\| \mathcal{P}_{\sigma,\lambda}^\tau f \bigr\|_{\mathcal E^p_\lambda}
 \leqslant  \, \gamma_\lambda  \bigl\| f \bigr\|_{L^p(K/M,\sigma)}
\end{equation*}

$(2)$  If $n$ 	is odd, then $\mathcal{P}_{\sigma_{n-1}^\pm,\lambda}^{\tau_n}$ is a topological isomorphism of the space $L^p(K/M, \sigma_{n-1}^\pm)$ onto the space $\mathcal{E}_{\pm,\lambda}^p(G/K,\tau_n)$.
Furthermore, there exists a positive constant $\gamma_\lambda$ such that, for every  $f\in L^p(K/M,\sigma)$ we have
\begin{equation*}
\sqrt{2}\, |c^\pm(\lambda,\tau_n)| \|f\|_{L^p(K/M,\sigma)}
\leqslant
 \bigl\| \mathcal{P}_{\sigma,\lambda}^\tau f \bigr\|_{\mathcal E^p_{\pm,\lambda}}
 \leqslant \sqrt{2}  \, \gamma_\lambda  \bigl\| f \bigr\|_{L^p(K/M,\sigma)}
\end{equation*}
\end{theorem}

The rest of the paper is devoted to the proof of the above statement.   

\section{Intermediate results}
 In the light of the Convention \ref{conv}, let  $\tau=\tau_n$ or $\tau_n^\pm$ and let $\sigma\in\widehat{M}(\tau)$.
\begin{proposition}\label{propo-princ}  For any $\lambda\in\mathbb C$ with $\Re(i\lambda) >0$,  there exists a positive constant $\gamma_\lambda$ such that, for every  $f\in L^p(K/M,\sigma)$ we have
\begin{equation*}
\left(\int_K \bigl\| \mathcal{P}_{\sigma,\lambda}^\tau f(ka_t) \bigr\|_{\tau}^p\, {\rm d}k\right)^{1/p} \leqslant \fsl\kappa  \, \gamma_\lambda   \, {\rm e}^{(\Re(i\lambda) -\rho)t} \bigl\| f \bigr\|_{L^p(K/M,\sigma)}
\end{equation*}
where $\kappa $ is given by  \autoref{kappa}.  
\end{proposition}  
\begin{proof} 
By the definition \autoref{Poisson} of the Poisson integrals, we have 
\begin{eqnarray*} 
\Vert \mathcal{P}_{\sigma,\lambda}^\tau f(ka_t)\Vert_{ V_\tau}
&\leq&\fsl\kappa  \int_K {\rm e}^{-(\Re(i\lambda)+\rho)H(a_t^{-1}k^{-1}h)}\Vert \iota_\sigma^\tau(f(h))\Vert_{ V_\tau }{\rm d}h,\\
&\leq&\fsl\kappa  \int_K {\rm e}^{-(\Re(i\lambda)+\rho)H(a_t^{-1}k^{-1}h)}\Vert   f(h)\Vert_{V_\sigma}  {\rm d}h \\
&=&\fsl\kappa \; {\bf e}_{\lambda,t}  \ast \Vert f(\cdot)\Vert_{V_\sigma}(k),
 \end{eqnarray*}
where ${\bf e}_{\lambda,t}(k)={\rm e}^{-(\Re(i\lambda)+\rho)H(a_t^{-1}k^{-1})}$ and $\ast$ is the convolution product over $K$.
To conclude, we use  Young's inequality
and the fact that for  $\Re(i\lambda)>0$ we have
\begin{eqnarray*}
\Vert{\bf e}_{\lambda,t}\Vert_{L^1(K/M,\,\sigma)}&=&\int_K {\rm e}^{-(\Re(i\lambda)+\rho)H(a_t^{-1}k^{-1})}{\rm d}k\\
 & =& \phi_{-i\Re(i\lambda)}^{(\rho-\frac{1}{2},-\frac{1}{2})}(t), \\
 &=&e^{(\Re(i\lambda)-\rho)t}\left(c_{\rho-\frac{1}{2},-\frac{1}{2}}(-i\Re(i\lambda))+ o(1)\right) \; \text{as} \; t\to\infty,
  \end{eqnarray*}
where $\phi_\nu^{(\alpha,\beta)}$ is the Jacobi function  \autoref{jacobi} and $c_{\alpha,\beta}(\lambda)$ is the constant \autoref{simple}. For the last identity, we refer to \autoref{est-Jacobi}.   
\end{proof}

Let $\overline{N}=\theta(N)$, where $\theta$ is the Cartan involution of $G$ corresponding to \autoref{kp}. Define   the generalized Harish-Chandra $c$-function by
\begin{equation}\label{c-function}
\mathbf c(\lambda,\tau)=\int_{\overline{N}}{\rm e}^{-(i\lambda+\rho)H(\overline{n})}\tau(\kappa(\overline{n})){\rm d}\overline{n}\in \mathrm{End}_M(V_\tau),
\end{equation}
where ${\rm d}\overline n$ is the Haar measure on $\overline  N$ with the normalization
$$\int_{\overline N} {\rm e}^{-2\rho(H(\overline n))} {\rm d}\overline n=1.$$
It is well known that the integral \autoref{c-function} converges for $\lambda\in \mathbb C$ such that \(\Re(i\lambda)>0\) and it has a meromorphic continuation to \({\C}\).

\begin{proposition}[Fatou Lemma]\label{Fatou}
Let $\lambda\in\mathbb{C}$ such that $\Re(i\lambda)>0$.  Then 
\begin{align*}
\lim_{t\rightarrow \infty}{\rm e}^{(\rho-i\lambda)t}\mathcal{P}_{\sigma,\lambda}^\tau\, f(ka_t)=\fsl\kappa \, \mathbf c(\lambda,\tau)\iota_\sigma^\tau\,f(k)
\end{align*}
\begin{itemize}
\item[(i)] uniformly for $f\in C^\infty(K/M,\sigma)$;
\item[(ii)] in the $L^p(K,V_\sigma)$-sens for $f\in L^p(K/M,\sigma)$ with $1<p<\infty$.
\end{itemize}
\end{proposition}
\begin{proof}
The proof follows the same arguments as  in \cite[Theorem 4.3]{BBK}.
\end{proof}


 Since the restriction $\mathbf c(\lambda,\tau)_{|V_{\sigma}}$ commutes with the representation $\sigma$, then by Schur's lemma we have
 
$$\begin{aligned}
	\mathbf c(\lambda,\tau_n^\pm)&=c(\lambda,\tau_n^\pm) \mathrm{Id}_{V_{\sigma_{n-1}}}&\text{whenever}\; n \text{ even},\\
	\mathbf c(\lambda,\tau_n)&=c^+(\lambda,\tau_n) \mathrm{Id}_{V_{\sigma_{n-1}^+}} +\;
	c^-(\lambda,\tau_n) \mathrm{Id}_{V_{\sigma_{n-1}^-}} &\text{whenever}\; n \text{ odd}
\end{aligned}$$
for some complex coefficients $c(\lambda,\tau_n^\pm)$ and $c^\pm(\lambda,\tau_n).$ To compute explicitly these scalar components, we will study  the asymptotic behaviour of the so-called $\tau$-spherical functions.

A continuous function \(F\colon G\rightarrow \mathrm{End}(V_\tau)\) is called elementary \(\tau\)-spherical if \(F\) satisfies
\begin{itemize}
\item[$(a)$]  $F(k_1gk_2)=\tau(k_2)^{-1}F(g)\tau(k_1^{-1}),$  
\item[$(b)$] \(F\) is a  joint-eigenfunction for $D(G, \tau)$ with $F(e)={\rm Id}.$
\end{itemize}

In view of the Convention \ref{conv}, for $\sigma\in \widehat M (\tau)$ and $\lambda\in \mathbb C,$ we consider  the function   $\Phi_{\sigma}^\tau(\lambda,\cdot) \colon G\rightarrow  \mathrm{End}(V_\tau) $  defined by 
\begin{align*}
\Phi_{\sigma}^\tau(\lambda,g)v=\fsl\kappa ^2  \int_K {\rm e}^{-(i\lambda+\rho)H(g^{-1}k)}\tau(\kappa(g^{-1}k)) \iota^\tau_\sigma(\pi^\tau_\sigma(\tau(k^{-1})v)){\rm d}k,\;\; \forall v\in V_\tau
\end{align*}
is a $\tau$-spherical function.
Using the Cartan decomposition $G=KAK$, it is clear that $\Phi_{\sigma}^\tau(\lambda,\cdot)$ is completely determined by its restriction to $A$. Since $A$ and $M$ commute, $\Phi_{\sigma}^\tau(\lambda,a_t)\in \End_M(V_\sigma)$  for all $a_t\in A$. Hence,  by Schur's lemma $\Phi_{\sigma}^\tau(\lambda,a_t)$ is scalar on each $M$-type $V_\sigma$  of $V_\tau$. That is
$$\begin{aligned}
	\Phi_{\sigma_{n-1}}^{\tau_n^\pm}(\lambda,a_t)&=\varphi^\pm(\lambda,t) \mathrm{Id}_{V_{\sigma_{n-1}}},& \text{whenever}\; n \;\text{is even},\\
	\Phi_{\sigma_{n-1}^\pm}^{\tau_n}(\lambda,a_t)&=\varphi_{\pm}^+(\lambda,t) \mathrm{Id}_{V_{\sigma_{n-1}^+}} +\;
	\varphi_{\pm}^-(\lambda,t) \mathrm{Id}_{V_{\sigma_{n-1}^-}},&\text{whenever}\; n \;\text{is odd}.
\end{aligned}$$
In \cite{CP}, the scalar components  $\varphi^\pm$, $\varphi_\pm^+$ and $\varphi_\pm^-$  are given in terms of the Jacobi function 
 \begin{equation}\label{jacobi}
 \phi_\lambda^{(\alpha,\beta)}(t)={}_2F_1\left(\frac{i\lambda+\alpha+\beta+1}{2},\frac{-i\lambda+\alpha+\beta+1}{2};\alpha+1; -\sinh^2 t\right)  
\end{equation}
where  $\alpha, \beta, \lambda\in \mathbb{C}$ with $\alpha\neq -1,-2, \ldots$,  (see, e.g. \cite{Ko}).
 

\begin{theorem}[{see \cite[Theorem 5.4]{CP}}]\label{CP-tau-sph} We have:

$(1)$  When $n$ is even,   
$$
\varphi^{\pm}(\lambda, t)=\left(\cosh\frac{t}{2}\right) \phi_{2 \lambda}^{(n / 2-1, n / 2)}\left(\frac{t}{2}\right).
$$

$(2)$  When $n$ is odd,  
\begin{eqnarray*}
 \varphi_{\pm}^{+}(\lambda, t) &=\left(\cosh \frac{t}{2}\right) \phi_{2 \lambda}^{(n / 2-1, n / 2)}\left(\frac{t}{2}\right) \pm i \frac{2 \lambda}{n}\left(\sinh \frac{t}{2}\right) \phi_{2 \lambda}^{(n / 2, n / 2-1)}\left(\frac{t}{2}\right), \\
  \varphi_{\pm}^{-}(\lambda, t) &=\left(\cosh \frac{t}{2}\right) \phi_{2 \lambda}^{(n / 2-1, n / 2)}\left(\frac{t}{2}\right) \mp i \frac{2 \lambda}{n}\left(\sinh \frac{t}{2}\right) \phi_{2 \lambda}^{(n / 2, n / 2-1)}\left(\frac{t}{2}\right).
\end{eqnarray*}
	
\end{theorem}


Next, we will compute the scalar components of the Harish-Chandra $c$-function $\mathbf c(\lambda,\tau)$.
\begin{proposition}\label{scal-comp-c}
We have:

$(1)$  When $n$ is even,  
\begin{equation*}
c(\lambda,\tau_n^+)= c(\lambda,\tau_n^-)=2^{n-i2\lambda}\frac{\Gamma(n/2)\Gamma(i2\lambda)}{\Gamma(i\lambda+n/2)\Gamma(i\lambda)}.
\end{equation*}

$(2)$ When $n$ is   odd,
\begin{equation*}	 
c^+(\lambda,\tau_n)=c^-(\lambda,\tau_n)=  2^{n-1-i2\lambda}\frac{\Gamma(n/2)\Gamma(i2\lambda)}{\Gamma(i\lambda+n/2)\Gamma(i\lambda)}.
\end{equation*}

\end{proposition}

 \begin{proof} Let us consider the case when $n$ is odd. 
It is well know that for 
  $\Re(i\lambda)>0$,  the Jacobi function satisfies 
\begin{equation}\label{est-Jacobi}
\phi_\lambda^{(\alpha,\beta)}(t)={\rm e}^{(i\lambda-\alpha-\beta-1)t}(c_{\alpha,\beta}(\lambda)+o(1))\,\; \text{as}\, \,  t\rightarrow \infty 
\end{equation}
 where 
\begin{equation}\label{simple}
  c_{\alpha,\beta}(\lambda)=\frac{2^{-i\lambda+\alpha+\beta+1}\Gamma(\alpha+1)\Gamma(i\lambda)}{\Gamma\left(\frac{i\lambda+\alpha+\beta+1}{2}\right)\Gamma\left(\frac{i\lambda+\alpha-\beta+1}{2}\right)}.
\end{equation} 
In the light of \autoref{simple}, one may check   the identity
\begin{equation}\label{recurrence}
 c_{\frac{n}{2},\frac{n}{2}-1}(2\lambda)=\frac{n}{2i\lambda}c_{\frac{n}{2}-1,\frac{n}{2}}(2\lambda).	
\end{equation}
 Let us first consider the case of $\Phi_{\sigma_{n-1}^+}^{\tau_n}$. Since $$\Phi_{\sigma_{n-1}^+}^{\tau_n}(\lambda,a_t)=\varphi_+^+(\lambda,t)\mathrm{Id}_{V_{\sigma_{n-1}^+}}+\varphi_-^+(\lambda,t)\mathrm{Id}_{V_{\sigma_{n-1}^-}},$$
then by Theorem \ref{CP-tau-sph} together with  \autoref{est-Jacobi}, \autoref{simple} and \autoref{recurrence},  we deduce that 
\begin{equation}\label{spherical3}\begin{split}
\lim_{t\rightarrow \infty}{\rm e}^{(\rho-i\lambda)t}\varphi^+_+(\lambda,t)  &=c_{\frac{n}{2}-1,\frac{n}{2}}(2\lambda)
 =2^{n-i2\lambda}\frac{\Gamma(n/2)\Gamma(i2\lambda)}{\Gamma(i\lambda+n/2)\Gamma(i\lambda)}
\end{split}\end{equation}
and 
\begin{align*}
\lim_{t\rightarrow \infty}{\rm e}^{(\rho-i\lambda)t}\varphi_+^-(\lambda,t)   =0.
\end{align*}
On the other hand,  $\Phi_{\sigma^+_{n-1}}^{\tau_n}(\lambda, ka_t)$ can be written in terms of the Poisson transform as  
$\Phi_{\sigma^+_{n-1}}^{\tau_n}(\lambda, ka_t) v=\fsl\kappa  \,\mathcal P_{\sigma_{n-1}^+}^{\tau_n}\big( \pi_{\sigma_{n-1}^+}^{\tau_n}\left[\tau_n(k^{-1}) v \right]\big)(a_t)$. Thus, by   Proposition \ref{Fatou}   we get
$$
\begin{aligned}
\lim_{t\rightarrow \infty}{\rm e}^{(\rho-i\lambda)t}\Phi_{\sigma_{n-1}^+}^{\tau_n}(\lambda,a_t)
&=\fsl\kappa ^2 c^+(\lambda,\tau_n) \mathrm{Id}_{V_{\sigma_{n-1}^+}}\\
&= 2 c^+(\lambda,\tau_n)\mathrm{Id}_{V_{\sigma_{n-1}^+}}.\\
\end{aligned}$$
Comparing this with \autoref{spherical3} we obtain 
$$c^+(\lambda,\tau_n)=2^{n-1-i2\lambda}\frac{\Gamma(n/2)\Gamma(i2\lambda)}{\Gamma(i\lambda+n/2)\Gamma(i\lambda)}.
$$ 
Now if we consider the case of $\Phi_{\sigma_{n-1}^-}^{\tau_n}$,  then by 
  the same arguments as  in the proof of $c^+(\lambda,\tau_n)$ we get
  $$c^-(\lambda,\tau_n)=2^{n-1-i2\lambda}\frac{\Gamma(n/2)\Gamma(i2\lambda)}{\Gamma(i\lambda+n/2)\Gamma(i\lambda)}.
$$
The case when $n$ is even is handled in the same way.
\end{proof}


\begin{proposition}
$(1)$  If $n$ is even, then there exists a positive constant $\gamma_\lambda$ such that, for every  $f\in L^p(K/M,\sigma)$ we have
\begin{equation*}
 |c(\lambda,\tau_n^\pm)| \|f\|_{L^p(K/M,\sigma)}
\leqslant
 \bigl\| \mathcal{P}_{\sigma,\lambda}^\tau f \bigr\|_{\mathcal E^p_\lambda}
 \leqslant  \, \gamma_\lambda    \bigl\| f \bigr\|_{L^p(K/M,\sigma)}
\end{equation*}

$(2)$  If $n$ is odd, then  there exists a positive constant $\gamma_\lambda$ such that, for every  $f\in L^p(K/M,\sigma)$ we have
\begin{equation*}
\sqrt{2}\, |c^\pm(\lambda,\tau_n)| \|f\|_{L^p(K/M,\sigma)}
\leqslant
 \bigl\| \mathcal{P}_{\sigma,\lambda}^\tau f \bigr\|_{\mathcal E^p_{\pm,\lambda}}
 \leqslant  \sqrt{2}\, \gamma_\lambda    \bigl\| f \bigr\|_{L^p(K/M,\sigma)}
\end{equation*}

\end{proposition}
\begin{proof} The proof is similar to  \cite[Proposition 4.4]{BBK}.
\end{proof}
By abuse of notation, we will denote the scalar components $c(\lambda,\tau_n^\pm)$ and $c^\pm(\lambda,\tau_n)$  by $c(\lambda, \tau)$ where the meaning is clear from the context.

\section{Proof of the main theorem for $p=2$}
\subsection{Auxilliary results}


Recall the Convention \ref{conv} and let $(\sigma,V_\sigma)\in \widehat{M}(\tau)$   of dimension $d_{\sigma}$.  Let $\widehat K(\sigma)\subset \widehat K$ be the subset of unitary equivalence classes of irreducible representations containing    $\sigma$ upon restriction to $K$. Consider an element $(\delta,V_\delta)$   in $\widehat{K}(\sigma).$
   From   \cite{BS}  or \cite{IT}  it   follows     that $\sigma$ occurs in $\delta_{|M}$ with multiplicity one, and therefore $\dim \mathrm{Hom}_M(V_\delta, V_\sigma)=1$.  Choose  the orthogonal projection $P_\delta : V_\delta\to V_\sigma$ to be   a generator of $\mathrm{Hom}_M(V_\delta, V_\sigma)$.
 
Let   $\{v_j: j=1,\ldots, d_\delta=\dim V_\delta\}$ be an orthonormal basis for $V_\delta$.
 Then  the set of functions 
$\{  \phi^\delta_j: 1\leqslant j\leqslant d_\delta, \; \delta\in \widehat{K}(\sigma)\}$ defined by 
$$k\mapsto \phi^\delta_j(k)=P_\delta(\delta(k^{-1})v_j) $$ 
is  an orthogonal basis of the space $L^2( K/M; \sigma )$, see, e.g., \cite{wallach}. 
Hence, the  Fourier series expansion of each   $ f$ in $L^2( K/M; \sigma)$ is given by   
$$f(k)=\sum_{\delta\in\widehat{K}(\sigma)}\sum_{j=1}^{d_\delta} a^\delta_{j} \phi^{\delta}_j(k),$$
with 
\begin{equation*}
\displaystyle \Vert f\Vert^2_{L^2(K/M;\, \sigma)}=\sum_{\delta\in\widehat{K}(\sigma)} \frac{d_\sigma}{d_\delta} \sum_{j=1}^{d_\delta}\mid a^\delta_{j}\mid^2.   
\end{equation*}

Define the following  Eisenstein integrals $\Phi_{\lambda,\delta}$ by
\begin{equation}\label{Eisen}
\Phi_{\lambda,\delta}(g)(v)=\fsl\kappa  \int_K {\rm e}^{-(i\lambda+\rho)H(g^{-1}k)}\tau(\kappa(g^{-1}k)) \iota_\sigma^\tau P_\delta(\delta(k^{-1}) v){\rm d}k,  
\end{equation}
for $g\in G$ and $v\in V_\delta$. One may check that   \(\displaystyle \Phi_{\lambda,\delta}(k_1gk_2)=\tau(k_2^{-1})\Phi_{\lambda,\delta}(g)\delta(k_1^{-1})\) for every  \(g\in G\) and \(k_1,k_2\in K\).

\begin{lemma}\label{key-lemma} We have
 \begin{equation}\label{key1}
   \sup _{t>0} {\rm e}^{(\rho-\Re(i\lambda))t} \|\Phi_{\lambda,\delta} (a_t)\|_{\rm{HS}} \leqslant \fsl\kappa \gamma_\lambda    \| P_\delta\|_{\rm{HS}}= \fsl\kappa \gamma_\lambda    \sqrt{d_\sigma},
   \end{equation}
 and

\begin{equation}\label{key2}
\lim_{t\to\infty}{\rm e}^{2(\rho-\Re(i\lambda))t} \|\Phi_{\lambda,\delta}(a_t)\|^2_{\rm HS}=\fsl\kappa^2 |c(\lambda,\tau)|^2 d_{\sigma}.
\end{equation}
\end{lemma}
\begin{proof} The proof follows the same lines as in \cite[Lemma 5.3, Lemma 5.4]{BBK}, so we omit it. 
\end{proof}

A functional on $G\times_{P} V_{\sigma }$ is  a  linear form $T$ on 
$\displaystyle C^\infty(G/P; \sigma_{\overline{\lambda}})$.
For a such functional    $T$, we define $\widetilde{\mathcal{P}_{\sigma,\lambda}^\tau}(T)$ by
\begin{equation}\label{nkolik} 
 \langle v, \widetilde{\mathcal{P}_{\sigma,\lambda}^\tau} T(g)\rangle_{V_\tau } 
 =\kappa  (T , \pi_\sigma^\tau L_g\Phi_\lambda v),\;\,\; \forall v\in V_\tau 
  \end{equation}
where
$L_g$ is the left regular action, and $\Phi_\lambda \colon G\to \mathrm{End}(V_{\tau })$ is given by 
\begin{equation}\label{phi-lam}
\Phi_\lambda(g)={\rm e}^{(i\overline{\lambda}-\rho)H(g)}\tau^{-1}(\kappa(g)).
\end{equation}
Notice that   $\Phi_\lambda(g^{-1}k)^*=P_{\sigma,\lambda}^\tau(g,k)$, where
$P_{\sigma,\lambda}^\tau \colon G\times K\to \mathrm{End}(V_{\tau })$ is the Poisson kernel given by
 \begin{equation}\label{Poisson-ker}
 P_{\sigma,\lambda}^\tau(g,k)=  {\rm e}^{-(i\lambda+\rho)H(g^{-1}k)}\tau(\kappa(g^{-1}k)).
 \end{equation}
 If $T=T_f$ is a functional given by a smooth function   $f\in C^\infty(G/P;\sigma_{\lambda})$,  then  by using \autoref{nkolik} together with  the facts that $\Phi_\lambda(g^{-1}k)^*$ coincides with the Poisson kernel and that the adjoint of $\pi_\sigma^\tau$ is the embedding $\iota_\sigma^\tau,$ one may check that 
 \begin{equation}\label{T}
\widetilde{\mathcal{P}_{\sigma,\lambda}^\tau}(T_f) = \mathcal{P}_{\sigma,\lambda}^\tau(f).
 \end{equation}
   

 
 \subsection{Proof of the main Theorem for $p=2$}
The necessary condition   follows from Theorem \ref{CP}  and Proposition \ref{propo-princ}. 

 On the other hand, 
let $F$ in  $\mathcal{E}^2_{\lambda}(G/K ; \tau_n^\pm)$ (when $n$ is even) or in $\mathcal{E}^2_{\pm,\lambda}(G/K ; \tau_n)$ (when $n$ is odd), 
 with the $K$-type expansion 
$$
F(g)=\sum_{\delta \in \widehat{K}(\sigma)} F_{\delta}(g)
$$
 then, by \cite[Corollary 10.8]{Yang}, there exists a $K$-finite vector $f_{\delta}$ in $C^{\infty}\left(G / P ; \sigma_{\lambda}\right)$ such that $F_{\delta}=\mathcal{P}_{\sigma,\lambda}^{\tau} f_{\delta}$. This implies
$$
f_{\delta}(k)=\sum_{j=1}^{d_{\delta}} a_{j}^{\delta} P_{\delta}\left(\delta\left(k^{-1}\right) v_{j}\right).
$$
It follows from 
  \cite[Proposition 5.1]{BBK} that the  functional $T$  on $C^\infty(G/P;\sigma_{\overline \lambda})$ defined by
\begin{equation}\label{def-T}
(T,\varphi)=   \sum_{\delta\in\widehat{K}(\sigma)}   \sum_{j=1}^{d_\delta}\overline{a^\delta_{j}}\int_K \ \langle \varphi(k),P_\delta(\delta(k^{-1})v_j)\rangle_{V_\sigma}\, {\rm d}k,
\end{equation}
satisfies 
$F=\widetilde{\mathcal{P}_{\sigma,\lambda}^\tau}T$. Therefore,  
\begin{align*}
F(g)=\fsl\kappa \sum_{\delta\in\widehat{K}(\sigma)}  \sum_{j=1}^{d_\delta}a^\delta_{j}\int_K {\rm e}^{-(i\lambda+\rho)H(g^{-1}k)}\tau(\kappa(g^{-1}k)) \iota_\sigma^\tau P_\delta(\delta(k^{-1})v_j){\rm d}k.
\end{align*} 
Define $\Phi_{\lambda,\delta}$ by
\begin{equation}\label{Eisen}
\Phi_{\lambda,\delta}(g)(v)=\fsl\kappa  \int_K {\rm e}^{-(i\lambda+\rho)H(g^{-1}k)}\tau(\kappa(g^{-1}k)) \iota_\sigma^\tau P_\delta(\delta(k^{-1}) v){\rm d}k,  
\end{equation}
for $g\in G$ and $v\in V_\delta$. One may check that   \(\displaystyle \Phi_{\lambda,\delta}(k_1gk_2)=\tau(k_2^{-1})\Phi_{\lambda,\delta}(g)\delta(k_1^{-1})\) for every  \(g\in G\) and \(k_1,k_2\in K\).
Further, using the covariance property of $\Phi_{\lambda,\delta}$ and Schur's orthogonality relations, we get
\begin{eqnarray*}
\int_K\Vert F(ka_t)\Vert_{\tau }^2{\rm d}k
&=&   \sum_{\delta\in \widehat{K}(\sigma)}   \frac{1}{d_\delta}\sum_{j=1}^{d_\delta} |a_{j}^\delta|^2  \tr\left(\Phi_{\lambda,\delta}(a_t)^\ast\Phi_{\lambda,\delta}(a_t)\right),\\
&=  &\sum_{\delta } \frac{1}{d_\delta}   \|    \Phi_{\lambda,\delta}(a_t) \|_{\rm{HS}}^2 \sum_j   |a^\delta_{j}|^2,
\end{eqnarray*}
where $\|\cdot \|_{\rm{HS}}$ is the Hilbert-Schmidt norm. 

 Let $\Lambda$ be a finite subset of $ \widehat K(\sigma)$, then 
\begin{eqnarray*}
 \sum_{\delta\in \Lambda}  \frac{1}{d_\delta} \sum_j   \|   a^\delta_{j} {\rm e}^{(\rho-i\lambda)t} \Phi_{\lambda,\delta}(a_t) \|_{\rm{HS}}^2
 & \leqslant &\displaystyle \sup_{t>0}{\rm e}^{2(\rho-\Re(i\lambda))t}\int_K\Vert F(ka_t)\Vert_{\tau }^2{\rm d}k,\\
 &=&\|F\|^2_{\mathcal E^2_{\lambda}} <\infty.
 \end{eqnarray*}
By \autoref{key2}
we have
 $$ \fsl\kappa ^2 |c(\lambda,\tau)|^2\sum_{\delta\in \Lambda}  \frac{d_\sigma}{d_\delta} 
 \sum_j      |  a^\delta_{j}|^2   \leqslant \|F\|^2_{\mathcal E^2_{\lambda}}.
 $$
Since the subset $\Lambda\subset \widehat K(\sigma)$ is arbitrary, it follows that  
\begin{align*}
\fsl\kappa ^2 |c(\lambda,\tau)|^2 \sum_{\delta\in \widehat{K}(\sigma)}\frac{d_\sigma}{d_\delta}\sum_{j}\vert  {a_{j}^\delta}\vert^2 \leqslant \parallel F\parallel^2_{\mathcal E^2_{\lambda}}<\infty.
\end{align*}
This shows that the functional   $ T(k)\sim  \sum_{\delta\in \widehat{K}(\sigma)}\sum_{j=1}^{d_\delta}  a_{j}^\delta P_\delta \delta(k^{-1})v_j$   defines a function \(f\in L^2(K/M;\sigma)\) and by \autoref{T} we deduce that $F=\mathcal{P}_{\sigma,\lambda}^\tau f$ with    
$$\fsl\kappa  |c(\lambda,\tau)| \Vert f\Vert_{L^2(K/M;\,\sigma)}\leqslant \Vert F \Vert_{\mathcal E^2_{\lambda}}.$$
This finishes  the proof of   the main Theorem \ref{main} for $p=2$.

\subsection{The inversion formula} We close this section by the following inversion formula which will be needed in the proof of   the main Theorem \ref{main} for arbitrary  $p$.

\begin{theorem}[Inversion formula]\label{inversion} Let $\tau=\tau_n$ or $\tau_n^\pm$ and $\sigma\in\widehat M(\tau)$. 
Assume  $\lambda\in \mathbb C$ with  $\Re(i\lambda)>0$.  Let $F$ be an element in 
$\mathcal{E}^{2}_{\lambda}(G/K ; \tau_n^\pm)$ (when $n$ is even) or in 
$\mathcal{E}^{2}_{\pm,\lambda}(G/K ; \tau_n)$ (when $n$ is odd), and let $f\in L^2(K/M;\sigma)$ such that $F=\mathcal P_{\sigma,\lambda}^\tau f.$ Then the following inversion formula holds in $L^2(K/M;\sigma)$
\begin{equation*}
f(k) =\fsl\kappa ^{-1} |c(\lambda, \tau)|^{-2}      \, \lim_{t\to\infty} e^{2(\rho- \Re(i\lambda) )t} \pi_\sigma^\tau \left(\int_K P_{\sigma,\lambda}^\tau(ha_t,k)^* F(ha_t)\, {\rm d}h\right),
\end{equation*}
 where $P_{\sigma,\lambda}^\tau(\cdot,\cdot)^*$ is adjoint of the Poisson kernel  \autoref{Poisson-ker} and  $\pi_\sigma^\tau$ is the projection of $V_\tau$ onto $V_\sigma$. 
\end{theorem}

\begin{proof} For the convenience of the reader we will outline the proof  which follows the same arguments as in \cite[Theorem 5.5]{BBK}.

Let $F$ be as in the statement.  By the main theorem for $p=2$, there exists a unique $f\in L^2(K/M;\sigma )$ such that $F=\mathcal{P}_{\sigma,\lambda}^\tau f$.  Using the notation introduced at the beginning of this section, we write $$f(k)=\sum_{\delta\in\widehat K(\sigma )}  \sum_{j=1}^{d_\delta} a^\delta_{j}P_\delta (\delta(k^{-1}))v_j.$$ Then
$$
F(ka_t)
= \sum_\delta \sum_{j} a^\delta_{j} \Phi_{\lambda,\delta} (a_t)\delta(k^{-1}) v_j,
$$
and therefore
$$
\int_K \Vert F(ka_t)\Vert_{\tau }^2\, {\rm d}k
=
\sum_\delta \sum_{j} \frac{|a^\delta_{j}|^2}{d_\delta} \| \Phi_{\lambda,\delta} (a_t) \|^2_{\rm HS}.
$$
Then, by lemma \ref{key-lemma},  we obtain
$$
\lim_{t\to\infty} {\rm e}^{2(\rho-\R(i\lambda))t} \int_K
 \| \mathcal{P}_{\sigma,\lambda}^\tau f(ka_t)\|_{\tau }^2\,{\rm d}k
 =\fsl\kappa ^2 |c(\lambda,\tau)|^2 \| f\|^2_{L^2(K/M;\sigma)},
$$
 which implies
 $$\lim_{t\to\infty}(g_t, \varphi)_{L^2(K/M;\,\sigma )}= ( f,   \varphi)_{L^2(K/M;\,\sigma )}  $$ for all $\varphi\in L^2(K/M;\sigma ),$
 where $g_t$ is the $V_{\sigma }$-valued function defined by 
 $$g_t(k)={\fsl\kappa ^{-1}}|c(\lambda,\tau)|^{-2} {\rm e}^{2(\rho-\Re(i\lambda))t} \pi_\sigma^\tau \int_K P_{\sigma,\lambda}^\tau(ha_t,k)^* F(ha_t)\,{\rm d}h.$$ 
 On the other hand, following the same arguments as in \cite[Theorem 5.5]{BBK}, the Fourier coefficients  $c_j^\delta(g_t)$ of $g_t$ are given by 
  \begin{eqnarray*}
  c_j^\delta(g_t)&=& {{d_\delta}\over {d_\sigma}} \int_K \langle g_t(k), P_\delta \delta(k^{-1})v_j\rangle_{V_\sigma} \,{\rm d}k\\
  &=&\fsl\kappa ^{-2} |c(\lambda, \tau)|^{-2} {\rm e}^{2(\rho-\Re(i\lambda))t} {{a_j^\delta}\over {d_\sigma}}  \Vert \Phi_{\lambda,\delta}(a_t) \Vert^2_{\rm HS}.
   \end{eqnarray*}
   Hence,
$$
  \|g_t\|^2_{L^2(K/M,\sigma)}
 = \left({\rm e}^{2(\rho-\Re(i\lambda))t}\fsl\kappa ^{-2} |  c(\lambda,\tau)|^{-2}\right)^2 
 \sum_\delta \frac{d_\sigma}{d_\delta} \sum_{j} \frac{1}{d_\sigma^2}       |a^\delta_{j}|^2 \Vert \Phi_{\lambda,\delta}(a_t) \Vert^4_{\rm HS}.
$$
Now, using \autoref{key2} we conclude that   
   \begin{eqnarray*}
  \lim_{t\to\infty} \|g_t\|^2_{L^2(K/M;\sigma)}
=  \sum_\delta \frac{d_\sigma}{d_\delta} \sum_{j} |a^\delta_{j}|^2
=  \|f\|^2_{L^2(K/M;\sigma)}.
  \end{eqnarray*}
The interchange of   $\lim_{t\to\infty}$ and $\sum_{\delta}$ is justified by the fact that 
  $$
  \sum_{\delta \in \widehat K(\sigma)}\frac{1}{d_\delta  } \sum_{j=1}^{d_\delta}    ({\rm e}^{2(\rho-\Re(i\lambda))t})^2   |a^\delta_{j}|^2 \| \Phi_{\lambda,\delta}(a_t) \|^4_{\rm HS},
  $$
  is uniformly  convergent, by \autoref{key1}. 
\end{proof}


 \section{Proof of the main theorem for $1<p<\infty$} 
 As in the case $p=2,$  the necessary condition follows from  Theorem \ref{CP} and Proposition \ref{propo-princ}. Let us turn our attention to the sufficiency condition.
 
Let $F$ be in the Hardy type space  and write $F(g)=\sum_i F_i(g) u_i$ where $\{u_i\}_i$ is an orthonormal basis of $V_\tau $. 
Consider   an approximation of the identity  $(\chi_m)_m$  in $C^\infty(K),$ and define the sequence $(F_{i,m})_m$ by 
$F_{i,m}(g)=\int_K\chi_m(h)F_i(h^{-1}g){\rm d}h$. As  $(\chi_m)_m$ is an approximation of the identity, it follows that $(F_{i,m})_m$ converges point-wise to $F_i$. On the other hand, define $F_m : G\to V_\tau $ by $F_m(g)=\sum_i F_{i,m}(g) u_i$. That is,
\begin{eqnarray*}
F_m(g)
&=& \int_K \chi_m(h) F(h^{-1} g) {\rm d}h,\\
\end{eqnarray*}
and we have $\|F_m(g) -F(g)\|^2_{\tau }   {\to 0}$ as ${m\to\infty}.$ Further, since the operators  $\fsl D$ and $\fsl D^2$  are $K$-invariant,     then $F_m$ belongs to either $\mathcal E_\lambda$ or $\mathcal E_{\pm, \lambda}$ according to whether  $n$ is even or odd (see \autoref{g1} and \autoref{g2}). 
Further,
\begin{eqnarray*}
 F_m(ka_t)&=&\int_K \chi_m(h) F(h^{-1} ka_t) {\rm d}h,\\
&=&\left( \chi_m\ast F^t\right)(k),
\end{eqnarray*}
where $F^t : K\to V_\tau $ is defined for any $t>0$ by $F^t(k) =F(ka_t)$. 
Since
$$ \| (\chi_m\ast F^t)(k)\|_{\tau }
\leqslant \int_K
 |\chi_m(h) |  \| F^t(h^{-1}k)\|_{\tau }   {\rm d}h=  |\chi_m (\cdot)|  \ast \| F^t(\cdot)\|_{\tau }(k) ,$$
then
\begin{eqnarray*}
 \| F_m^t\|_{L^p(K;V_\tau )}
 &\leqslant& \left\| |\chi_m (\cdot)|  \ast \| F^t(\cdot)\|_{\tau }\right\|_{L^p(K)}.
 \end{eqnarray*}
Applying  Young's inequalities to the right-hand side of the above inequality we get 
 \begin{eqnarray}\label{holder1}
 \| F_m^t\|_{L^p(K;V_\tau )} 
 &\leqslant & 
  \|F^t\|_{L^p(K;V_\tau )},
 \end{eqnarray}
 and
 \begin{eqnarray}\label{holder2}
 \| F_m^t\|_{L^2(K;V_\tau )} &\leqslant & \|\chi_m\|_{L^2(K)} \; \left\| \| F^t(\cdot)\|_{V_\tau }\right\|_{L^1(K)},\nonumber\\
 &\leqslant & \|\chi_m\|_{L^2(K)} \;  \| F^t\|_{L^p(K;V_{\tau })},
 \end{eqnarray}
since $p> 1$. The last inequality  says 
$$\sup_{t>0} {\rm e}^{(\rho-\Re(i\lambda))t} \left(\int_K \|F_m(ka_t)\|_{\tau }^2 {\rm d}k\right)^{1/2}\leqslant  \|\chi_m\|_{L^2(K)}   \| F\|_{\mathcal E_{\sigma,\lambda}^p}<\infty.$$
Hence, for every $m$,  $F_m\in \mathcal{E}^{2}_{\sigma,\lambda}(G/K ; \tau)$ and by the previous section (the case $p=2$),   there exists $f_m\in L^2(K/M;\sigma )$ such that $F_m=\mathcal{P}_{\sigma,\lambda}^\tau f_m$. 
To prove that  in fact $f_m\in L^p(K/M;\sigma )$ we follow a similar approach as in \cite{BBK}.    
According to the inversion Theorem \ref{inversion}  we have, for any $\varphi\in C(K/M;\sigma )$,
$$\int_K \langle f_m(k),\varphi(k)\rangle_{\sigma}{\rm d}k=\lim_{t\to \infty} 
\int_K \langle g_m^t(k), \varphi(k)\rangle_{\sigma} {\rm d}k,$$
where
$$g_m^t(k):=\fsl\kappa ^{-2} |c(\lambda,\tau)|^{-2} e^{2(\rho-\Re(i\lambda))t} \pi^\tau_\sigma\int_K P_{\sigma,\lambda}^\tau(ha_t,k)^* F_m(ha_t){\rm d}h.$$
Further,
 \begin{eqnarray*}
 &&\left|\int_K \langle g_m^t(k),\varphi(k)\rangle_{\sigma} {\rm d} k\right|\\
&&= \left|\fsl\kappa ^{-3}  |c(\lambda,\tau)|^{-2} e^{2(\rho-\Re(i\lambda))t} \int_K \langle F_m(ha_t), (\mathcal{P}_{\sigma,\lambda}^\tau \varphi)(ha_t) \rangle_{\tau } {\rm d}h\right|\\
&&\leqslant \fsl\kappa ^{-3}  |c(\lambda,\tau)|^{-2} e^{2(\rho-\Re(i\lambda))t}  \int_K \|F_m(ha_t)\|_{\tau } \|\mathcal{P}_{\sigma,\lambda}^\tau\varphi (ha_t)\|_{\tau } {\rm d}h.
 \end{eqnarray*}
By H\"older's inequality (with $\frac{1}{p}+\frac{1}{q}=1$), we deduce 
 \begin{eqnarray*}
&& \left|\int_K \langle g_m^t(k),\varphi(k)\rangle_{\sigma}{\rm d}k\right|\\
&& \leqslant \fsl\kappa ^{-3}  |c(\lambda,\tau)|^{-2} e^{2(\rho-\Re(i\lambda))t} \|F_m^t\|_{L^p(K;V_\tau)} \| (\mathcal{P}_{\sigma,\lambda}^\tau\varphi)^t\|_{L^q(K;V_\tau )},
 \end{eqnarray*}
 where $(\mathcal{P}_{\sigma,\lambda}^\tau\varphi)^t(k)=(\mathcal{P}_{\sigma,\lambda}^\tau\varphi)(ka_t)$.
Using  \autoref{holder1}  and Proposition \ref{propo-princ} we get
 \begin{eqnarray*}\label{eq-functional1}
 \left|\int_K \langle f_m(k),\varphi(k)\rangle_{\sigma}{\rm d}k\right| 
 &\leqslant & 
 \fsl\kappa ^{-2} \gamma_\lambda  |c(\lambda,\tau)|^{-2}   \|F_m\|_{\mathcal E_{\sigma,\lambda}^p} \|  \varphi\|_{L^q(K/M;\,\sigma )},\\
 &\leqslant& 
 \fsl \kappa ^{-2}\gamma_\lambda  |c(\lambda,\tau)|^{-2}   \|F\|_{\mathcal E_{\sigma,\lambda}^p} \|  \varphi\|_{L^q(K/M; \,\sigma )}.
 \end{eqnarray*}
Taking the supremum over all $\varphi\in C(K/M;\sigma )$ with $\|  \varphi\|_{L^q(K/M;\sigma )}=1$ gives
$$\|f_m\|_{L^p(K/M;\, \sigma )} \leqslant \fsl\kappa ^{-2}  \gamma_\lambda  |c(\lambda,\tau)|^{-2}  \|F\|_{\mathcal E_{\sigma,\lambda}^p} , $$
which implies that $f_m$  is indeed  in   $L^p(K/M;\sigma )$.\\

For every integer $m$, define the linear form $T_m$ on $L^q(K/M;\sigma  )$ by
$$T_m(\varphi)=\int_K\langle f_m(k),\varphi(k) \rangle_{\sigma} {\rm d}k.$$
The operator $T_m$ is continuous and satisfies 
 \begin{eqnarray*}
 |T_m(\varphi)| 
  &\leqslant& \fsl \kappa ^{-2}\gamma_\lambda   |c(\lambda,\tau)|^{-2}   \|F\|_{\mathcal E_{\sigma,\lambda}^p} \|  \varphi\|_{L^q(K/M;\, \sigma )}.
  \end{eqnarray*}
This shows that  the sequence $(T_m)_m$ is uniformly bounded in $L^q(K/M;\sigma )$, with
$$\sup_m \|T_m\|_{\rm{op}} \leqslant \fsl\kappa ^{-2} \gamma_\lambda   |c(\lambda,\tau)|^{-2}   \|F\|_{\mathcal E_{\sigma,\lambda}^p}\;,$$
where $\| \cdot \|_{\rm{op}}$ stands for the operator norm.
By  Banach-Alaouglu-Bourbaki's theorem   there exists a subsequence of bounded operators $(T_{m_j})_j$ which converges to a bounded operator $T$ under the weak-star topology, with
$$  \|T\|_{\rm{op}} \leqslant \fsl\kappa ^{-2}\gamma_\lambda   |c(\lambda,\tau)|^{-2}   \|F\|_{\mathcal E_{\sigma,\lambda}^p}\;.$$
Therefore,   Riesz's representation theorem implies  the existence of  a unique $f\in L^p(K/M;\sigma )$ such that
$$T(\varphi) =\int_K \langle   \varphi(k), f(k) \rangle_{\sigma} {\rm d}k.$$
For an arbitrary given $v\in V_\tau $, taking   $\varphi(k)=\varphi_g(k)= P_{\sigma,\lambda}^\tau(g,k)^* v$ we obtain 
\begin{equation}\label{fin}
T(\varphi_g) =\langle v, \mathcal{P}_{\sigma,\lambda}^\tau f(g)\rangle_{\tau}.
\end{equation}
 On the other hand, we have
$$T_{m_j}(\varphi_g) =\langle v, \mathcal{P}_{\sigma,\lambda}^\tau f_{m_j}(g)\rangle_{\tau} = \langle v, F_{m_j}(g)\rangle_{\tau}.$$ 
After taking the limit  $j\to\infty$,  the above identity together with \autoref{fin}  imply 
$F(g)=\mathcal{P}_{\sigma,\lambda}^\tau f (g)$ for every $g\in G$. This finishes the proof of the main Theorem \ref{main} for every $1<p<\infty.$

 \pdfbookmark[1]{References}{ref}

\end{document}